\documentclass{amsart}
\usepackage{amssymb}
\numberwithin{equation}{section}
\def\N{\mathbb N}
\def\R{\mathbb R}
\def\S{\mathbb S}
\providecommand{\norm}[1]{\left\lVert#1\right\rVert}
\newcommand{\ip}[2]{\left\langle #1, #2 \right\rangle}

\providecommand{\abs}[1]{\left\lvert#1\right\rvert}
\DeclareMathOperator*{\Argmin}{argmin}
\DeclareMathOperator*{\Argmax}{argmax}
\DeclareMathOperator{\Fix}{\mathcal{F}}
\DeclareMathOperator{\AC}{\mathcal{A}}
\newcommand{\sk    }[1]{\left( #1 \right)}
\newcommand{\ck    }[1]{\left\{#1 \right\}}
\newcommand{\CAT}{\textup{CAT}}
\newcommand{\opclitvl}[2]{\left]#1, #2\right]}
\newcommand{\clopitvl}[2]{\left[#1, #2\right[}
\newcommand{\openitvl}[2]{\left]#1, #2\right[}
\theoremstyle{plain}
\newtheorem{theorem}{Theorem}[section]
\newtheorem{lemma}[theorem]{Lemma}

\newtheorem{corollary}[theorem]{Corollary}

\theoremstyle{definition}

\theoremstyle{remark}
\newtheorem{remark}[theorem]{Remark}
\allowdisplaybreaks
\title[The proximal point algorithm in geodesic spaces] 
{The proximal point algorithm in geodesic spaces with curvature bounded above}
\author[Y.~Kimura]{Yasunori~Kimura}
\address[Y.~Kimura]
{Department of Information Science, 
Toho University, 
Miyama, Funabashi, Chiba 274-8510, Japan}
\email{yasunori@is.sci.toho-u.ac.jp}
\author[F.~Kohsaka]{Fumiaki~Kohsaka}
\address[F.~Kohsaka]
{Department of Mathematical Sciences, Tokai University, 
Kitakaname, Hiratsuka, Kanagawa 259-1292, Japan}
\email{f-kohsaka@tsc.u-tokai.ac.jp}

\begin{document}
\begin{abstract}
We investigate the asymptotic behavior of sequences generated 
by the proximal point algorithm for convex functions 
in complete geodesic spaces with curvature bounded above.  
Using the notion of resolvents of such functions,  
which was recently introduced by the authors, 
we show the existence of minimizers of convex functions 
under the boundedness assumptions on such sequences 
as well as the convergence of such sequences 
to minimizers of given functions. 
\end{abstract}
\maketitle

\section{Introduction}\label{sec:intro}

The aim of this paper is to obtain the following 
two results on the asymptotic behavior of 
sequences generated by the proximal point algorithm for convex functions 
in complete $\CAT(1)$ spaces.  

\begin{theorem}\label{thm:ppa-CAT1-exist}
 Let $X$ be an admissible complete $\CAT(1)$ space, 
 $f$ a proper lower semicontinuous convex 
 function of $X$ into $\opclitvl{-\infty}{\infty}$, 
 $\{\lambda _n\}$ a sequence of positive real numbers 
 such that $\sum _{n=1}^{\infty} \lambda _n =\infty$, 
 and $\{x_n\}$ a sequence defined by 
 $x_1\in X$ and 
 \begin{align}\label{eq:ppa-CAT1}
  x_{n+1} = \Argmin_{y\in X} 
 \ck{f(y)+\frac{1}{\lambda_n} \tan d(y, x_{n})\sin d(y, x_{n})}
 \end{align}
 for all $n\in \N$. 
 Then the set $\Argmin_X f$ of all minimizers of $f$ is nonempty if and only if 
 $\{x_n\}$ is spherically bounded and $\sup_{n}d(x_{n+1},x_{n})<\pi / 2$. 
\end{theorem}

\begin{theorem}\label{thm:ppa-CAT1-conv}
 Let $X$, $f$, $\{\lambda_n\}$, and $\{x_n\}$ 
 be the same as in Theorem~\ref{thm:ppa-CAT1-exist} 
 and suppose that $\Argmin_X f$ is nonempty. 
 Then the following hold. 
 \begin{enumerate}
  \item[(i)] There exists a positive real number $C$ such that 
 \begin{align*}
  f(x_{n+1}) - \inf f(X) \leq \frac{C}{\sum_{k=1}^{n}\lambda _k}
 \bigl(1-\cos d(u, x_1)\bigr)
 \end{align*}
 for all $u\in \Argmin_X f$ and $n\in \N$; 
  \item[(ii)] $\{x_n\}$ is 
 $\Delta$-convergent to an element of $\Argmin_X f$. 
 \end{enumerate}
\end{theorem}

It should be noted that, in this paper, we say that 
a $\CAT(1)$ space $X$ is admissible if 
$d(v,v')<\pi/2$ for all $v,v'\in X$.  
We also say that a sequence $\{x_n\}$ in a $\CAT(1)$ space $X$ 
is spherically bounded if 
\begin{align}\label{eq:spherical-bdd}
 \inf_{y\in X} \limsup _{n\to \infty} d(y, x_n) < \frac{\pi}{2}.    
\end{align}

The proximal point algorithm, introduced by 
Martinet~\cite{MR0298899} and Rockafellar~\cite{MR0410483}, 
is an approximation method for finding 
a minimizer of a proper lower semicontinuous convex function 
$f$ of a real Hilbert space $X$ into $\opclitvl{-\infty}{\infty}$. 
This algorithm generates a sequence $\{x_n\}$ by $x_1\in X$ and 
\begin{align}\label{eq:ppa}
 x_{n+1} = \Argmin_{y\in X} 
 \ck{f(y)+\frac{1}{2\lambda _n} \norm{y-x_{n}}^2}
\end{align}
for all $n\in \N$, where $\{\lambda _n\}$ is a sequence 
of positive real numbers. 
It is well known that the right hand side of~\eqref{eq:ppa} 
consists of one point $p\in X$. 
We identify the set $\{p\}$ with $p$ in this case. 
Using the resolvent $J_{f}$ of $f$ given by 
\begin{align}\label{eq:res-Hilbert}
 J_{f}x=\Argmin_{y\in X} \ck{f(y) 
+ \frac{1}{2} \norm{y-x}^2}
\end{align}
for all $x\in X$, we can write the scheme~\eqref{eq:ppa} 
as $x_{n+1}=J_{\lambda _n f}x_n$ 
for all $n\in \N$. 
See~\cite{MR2798533, MR2548424} for more details 
on convex analysis in Hilbert spaces.  

In 1976, Rockafellar~\cite[Theorem~1]{MR0410483} showed that 
if $\inf_{n}\lambda _n >0$, then 
the set $\Argmin_X f$ is nonempty if and only if $\{x_n\}$ is bounded, 
and that if $\Argmin_X f$ is nonempty, then 
$\{x_n\}$ is weakly convergent to an element of $\Argmin_X f$. 
In 1978, Brezis and Lions~\cite[Th\'eor\`eme~9]{MR491922} 
showed the weak convergence of 
$\{x_n\}$ to an element of $\Argmin_X f$ 
under a weaker condition that 
$\Argmin_X f$ is nonempty and 
$\sum_{n=1}^{\infty}\lambda _n =\infty$.   
Later, G{\"u}ler~\cite[Corollary~5.1]{MR1092735} 
found an example of $\{x_n\}$ in the Hilbert space $\ell^{2}$ 
which does not converge strongly. 

On the other hand, in 1995, Jost~\cite{MR1360608} 
generalized the concept of resolvent given by~\eqref{eq:res-Hilbert} 
in Hilbert spaces to that in more general complete $\CAT(0)$ spaces. 
According to~\cite[Section~2.2]{MR3241330},~\cite[Lemma~2]{MR1360608}, 
and~\cite[Section~1.3]{MR1651416}, 
if $f$ is a proper lower semicontinuous convex
function of a complete $\CAT(0)$ space $X$ into $\opclitvl{-\infty}{\infty}$, 
then the resolvent $J_f$ of $f$ given by 
\begin{align}\label{eq:res-CAT0}
 J_{f}x= \Argmin _{y\in X} \ck{f(y) + \frac{1}{2} d(y, x)^2}
\end{align}
for all $x\in X$ 
is a well defined single valued nonexpansive mapping of $X$ into itself. 
We also know that its fixed point set $\Fix (J_{f})$ 
is equal to $\Argmin _{X}f$. 
See~\cite{MR3241330, MR1451625, MR1652278} 
for more details on this concept. 

In 2013, Ba{\v{c}}{\'a}k~\cite[Theorem~1.4 and Remark~1.6]{MR3047087} 
generalized the result by 
Brezis and Lions~\cite[Th\'eor\`eme~9]{MR491922} 
to the complete $\CAT(0)$ space setting as follows.  
Note that $\Delta$-convergence is called 
weak convergence in~\cite{MR3047087}. 
\begin{theorem}[{\cite[Theorem~1.4 and Remark~1.6]{MR3047087}}]
 \label{thm:Bacak}
 Let $X$ be a complete $\CAT(0)$ space, 
 $f$ a proper lower semicontinuous convex 
 function of $X$ into $\opclitvl{-\infty}{\infty}$ 
 such that $\Argmin_X f$ is nonempty, 
 $J_{\lambda f}$ the resolvent of $\lambda f$ for each $\lambda >0$, 
 $\{\lambda _n\}$ a sequence of positive real numbers 
 such that $\sum _{n=1}^{\infty} \lambda _n =\infty$, 
 and $\{x_n\}$ a sequence defined by 
 $x_1\in X$ and 
$x_{n+1} = J_{\lambda _n f} x_n$ 
 for all $n\in \N$. 
 Then $\{x_n\}$ is $\Delta$-convergent 
 to an element of $\Argmin_X f$ and 
 \begin{align*}
  f(x_{n+1}) - \inf f(X) \leq \frac{1}{2\sum_{k=1}^{n}\lambda _k}
  d(u, x_1)^2
 \end{align*}
 for all $u\in \Argmin_X f$ and $n\in \N$. 
\end{theorem}

Recently, the authors~\cite{MR3463526} 
introduced the concept of 
resolvents of convex functions in complete $\CAT(1)$ spaces 
and studied the existence and approximation of 
fixed points of mappings related to this concept.  
Considering the geometric difference between 
$\CAT(0)$ and $\CAT(1)$ spaces, 
they replaced $d(y, x)^2/2$ in~\eqref{eq:res-CAT0}  
with $\tan d(y, x) \sin d(y, x)$ in the definition of resolvent below.  
According to~\cite[Theorems~4.2 and~4.6]{MR3463526}, 
if $f$ is a proper lower semicontinuous convex function 
of an admissible complete $\CAT(1)$ space $X$ 
into $\opclitvl{-\infty}{\infty}$,  
then the resolvent $R_{f}$ of $f$ given by 
\begin{align}\label{eq:resolvent-CAT1}
 R_{f}x =\Argmin_{y\in X} 
 \bigl\{
 f(y) +\tan d(y, x) \sin d(y, x)\bigr\}
\end{align}
for all $x\in X$ 
is a well defined single valued mapping of $X$ into itself such that 
\begin{align}\label{eq:FP-Min}
 \Fix (R_{f})=\Argmin_X f
\end{align}
and 
\begin{align}
 \begin{split}\label{eq:SFSN-resolvent}
  &\Bigl(C_x^2(1+C_y^2)C_y+C_y^2(1+C_x^2)C_x\Bigr) \cos d(R_{f}x,R_{f}y) \\
  &\quad \geq C_x^2(1+C_y^2)\cos d(R_{f}x,y)
 +C_y^2(1+C_x^2)\cos d(R_{f}y, x)
 \end{split}
\end{align}
for all $x,y\in X$, where $C_z=\cos d(R_{f}z,z)$ for all $z\in X$. 
Using this concept, we can write the scheme~\eqref{eq:ppa-CAT1} as 
\begin{align*}
 x_{n+1}=R_{\lambda _n f}x_n
\end{align*}
for all $n\in \N$. 
The function $t\mapsto \tan t \sin t$ 
used in~\eqref{eq:resolvent-CAT1}   
is obviously a strictly increasing, continuous, and convex function 
on $\clopitvl{0}{\pi/2}$ such that 
$\tan 0 \sin 0=0$ and $\tan t \sin t \to \infty$ as $t\uparrow \pi/2$.  
These properties are similar to those of 
the function $t\mapsto t^2$ on $\clopitvl{0}{\infty}$ 
used in~\eqref{eq:res-CAT0}.  
Note that the diameters of the model spaces $\S^2$ and $\R^2$ 
of $\CAT(1)$ and $\CAT(0)$ spaces 
coincide with $\pi$ and $\infty$, respectively 
and that the second order Maclaurin approximation of 
the function $t\mapsto \tan t \sin t$ 
is equal to $t\mapsto t^2$.  

This paper is organized as follows. 
In Section~\ref{sec:pre}, 
we recall some definitions and results needed in this paper. 
In Section~\ref{sec:fund}, 
we obtain some fundamental properties of 
resolvents of convex functions in $\CAT(1)$ spaces. 
In Section~\ref{sec:PPA}, after obtaining Theorem~\ref{thm:argmax}, 
a maximization theorem in $\CAT(1)$ spaces, 
we give the proofs of Theorems~\ref{thm:ppa-CAT1-exist} 
and~\ref{thm:ppa-CAT1-conv}. 
In Section~\ref{sec:cor}, we obtain three corollaries 
of Theorems~\ref{thm:ppa-CAT1-exist} 
and~\ref{thm:ppa-CAT1-conv}. 

\section{Preliminaries}\label{sec:pre}

Throughout this paper, we denote by 
$\N$ the set of all positive integers, 
$\R$ the set of all real numbers, 
and $\Fix(T)$ the set of all fixed points of a mapping $T$. 

A metric space $X$ with metric $d$ is said to be 
uniquely $\pi$-geodesic 
if for each $x,y\in X$ with $d(x,y)<\pi$, 
there exists a unique mapping $c$ of $[0,l]$ into $X$ such that 
$d\bigl(c(t), c(t')\bigr)=\abs{t-t'}$ for all $t,t'\in [0,l]$, 
$c(0)=x$, and $c(l)=y$, where $l=d(x,y)$. 
The mapping $c$ is called the geodesic 
from $x$ to $y$ and 
the set $[x,y]$, which is defined as the image of $c$, 
is called the geodesic segment between $x$ and $y$. 
We also denote by $\alpha x \oplus (1-\alpha)y$ 
the point $c\bigl((1-\alpha)d(x,y)\bigr)$ 
for each $\alpha \in [0,1]$. 

Let $H$ be a real Hilbert space with inner product 
$\ip{\,\cdot\,}{\,\cdot\,}$ and 
the induced norm $\norm{\,\cdot \,}$. 
We know that the unit sphere $S_H$ of $H$ 
is a complete metric space with 
the spherical metric $\rho_{S_H}$ defined by 
\begin{align*}
 \rho _{S_H}(x,y) =\arccos \ip{x}{y}
\end{align*}
for each $x,y\in S_H$. 
It is also known that $S_H$ is uniquely $\pi$-geodesic. 
For each distinct $x,y\in S_H$ such that $\rho _{S_H}(x,y) <\pi$, 
the unique geodesic $c$ from $x$ to $y$ 
is given by 
\begin{align*}
 c(t) = (\cos t) x + (\sin t)\cdot \frac{y-\ip{y}{x}x}{\norm{y-\ip{y}{x}x}}
\end{align*}
for all $t\in [0,\rho _{S_H}(x,y)]$. 
We denote by $\mathbb{S}^2$ the unit sphere of 
the three dimensional Euclidean space $\R^3$. 

It is known~\cite[Lemma~2.14 in Chapter~I.2]{MR1744486} 
that if $X$ is a uniquely $\pi$-geodesic space 
and $x_1,x_2,x_3$ are points in $X$ satisfying 
\begin{align}\label{eq:three_point_cond}
 d(x_1,x_2) + d(x_2,x_3) + d(x_3,x_1) < 2\pi, 
\end{align}
then there exist $\bar{x}_1,\bar{x}_2,\bar{x}_3\in \mathbb{S}^2$ 
such that 
\begin{align*}
 d(x_{i},x_{j})=\rho_{\mathbb{S}^2} (\bar{x}_i,\bar{x}_j)
\end{align*}
for all $i, j\in \{1,2,3\}$. 
The sets $\Delta$ and $\bar{\Delta}$ given by 
\begin{align*}
 \Delta =[x_1,x_2]\cup[x_2,x_3]\cup[x_3,x_1]
\quad \textrm{and} \quad 
  \bar{\Delta} =[\bar{x}_1,\bar{x}_2]\cup[\bar{x}_2,\bar{x}_3]
 \cup[\bar{x}_3,\bar{x}_1]
\end{align*}
are called the geodesic triangle with vertices 
$x_1,x_2,x_3$ and 
a comparison triangle for $\Delta$, 
respectively. 
A point $\bar{p}\in \bar{\Delta}$ is 
called a comparison point for $p\in \Delta$ 
if $p\in [x_i, x_j]$, $\bar{p}\in [\bar{x}_i, \bar{x}_j]$, 
and $d(x_i,p)=\rho_{\mathbb{S}^2}(\bar{x}_i, \bar{p})$ 
for some distinct $i, j \in \{1,2,3\}$.  

A uniquely $\pi$-geodesic space $X$ is called a $\CAT(1)$ space if 
\begin{align*}
 d(p, q)\leq \rho_{\mathbb{S}^2}(\bar{p},\bar{q})
\end{align*}
whenever $\Delta$ is a geodesic triangle 
with vertices $x_1,x_2,x_3 \in X$ satisfying~\eqref{eq:three_point_cond}, 
$\bar{\Delta}$ is a comparison triangle for $\Delta$, 
and $\bar{p}, \bar{q}\in \bar{\Delta}$ 
are comparison points for $p,q\in \Delta$, respectively. 
We know that all nonempty closed convex subsets of 
a real Hilbert space $H$, the space $(S_H,\rho_{S_H})$, 
and all complete $\CAT(0)$ spaces 
are complete $\CAT(1)$ spaces. 
The complete $\CAT(1)$ space $(S_H, \rho_{S_H})$ 
is particularly called a Hilbert sphere. 
See~\cite{MR1744486} 
for more details on geodesic spaces. 

The following lemma plays a fundamental role in the 
study of $\CAT(1)$ spaces. 

\begin{lemma}[{\cite[Corollary~2.2]{MR2927571}}]
\label{lem:KS-inequality}
 Let $X$ be a $\CAT(1)$ space and 
 $x_1, x_2, x_3$ points of $X$ such that~\eqref{eq:three_point_cond} 
 holds. 
 If $\alpha \in [0,1]$, then 
 \begin{align*}
  \begin{split}
  &\cos d\bigl( \alpha x_1 \oplus (1-\alpha)x_2 , x_3\bigr) \sin d(x_1,x_2) \\
  &\geq \cos d(x_1, x_3) \sin \bigl(\alpha d(x_1,x_2)\bigr) 
  +\cos d(x_2, x_3) \sin \bigl((1-\alpha) d(x_1,x_2)\bigr). 
  \end{split}
 \end{align*}
\end{lemma}

We also know the following.  
\begin{lemma}[See the proof of~{\cite[Lemma~4.1]{MR3020188}}]
\label{lem:KS-inequality-Mid}
 Let $X$,  $x_1$, $x_2$, and $x_3$ be the same 
 as in Lemma~\ref{lem:KS-inequality}. 
 Then 
 \begin{align*}
  \cos d\Bigl( \frac{1}{2} x_1 \oplus \frac{1}{2}x_2, x_3\Bigr) 
    \cos \Bigl(\frac{1}{2}d(x_1,x_2)\Bigr) 
  \geq \frac{1}{2}\cos d(x_1, x_3) + \frac{1}{2}\cos d(x_2,x_3). 
 \end{align*}
\end{lemma}

\begin{lemma}[See, for instance,~{\cite[Lemma~2.3]{MR3463526}}]
\label{lem:CAT1-CMS}
 Let $X$, $x_1$, $x_2$, and $x_3$ be the same 
 as in Lemma~\ref{lem:KS-inequality}. 
 If $d(x_1,x_3) \leq \pi/2$, $d(x_2,x_3)\leq \pi/2$, and $\alpha \in [0,1]$, then 
 \begin{align*}
  \cos d\bigl(\alpha x_1 \oplus (1-\alpha) x_2, x_3\bigr) 
  \geq \alpha \cos d(x_1, x_3) + (1-\alpha) \cos d(x_2,x_3). 
 \end{align*}
\end{lemma}

Let $X$ be a $\CAT(1)$ space and $\{x_n\}$ a sequence in $X$. 
The asymptotic center $\AC\bigl(\{x_n\}\bigr)$ 
of $\{x_n\}$ is defined by 
\begin{align*}
 \AC\bigl(\{x_n\}\bigr) 
 =\ck{z\in X: \limsup_{n\to \infty} d(z, x_n) 
 =\inf_{y\in X}\limsup_{n\to \infty} d(y, x_n)}. 
\end{align*}
The sequence $\{x_n\}$ is said to be $\Delta$-convergent 
to an element $p\in X$ if 
\begin{align*}
 \AC\bigl(\{x_{n_i}\}\bigr)=\{p\} 
\end{align*}
for each subsequence $\{x_{n_i}\}$ of $\{x_n\}$. 
In this case, the point $p$ is called the $\Delta$-limit 
of $\{x_n\}$. 
If $\{x_n\}$ is $\Delta$-convergent to $p\in X$,  
then it is bounded and 
each subsequence of $\{x_n\}$ is $\Delta$-convergent to $p$. 
For a sequence $\{x_n\}$ in $X$, we denote by 
$\omega_{\Delta}\bigl(\{x_n\}\bigr)$ the set of all 
points $q\in X$ such that there exists a subsequence of 
$\{x_n\}$ which is $\Delta$-convergent to $q$. 
It is known~\cite[Proposition~4.1 and Corollary~4.4]{MR2508878} 
that if $X$ is a complete $\CAT(1)$ space 
and $\{x_n\}$ is a spherically bounded sequence in $X$, 
that is, it satisfies~\eqref{eq:spherical-bdd},  
then $\AC\bigl(\{x_n\}\bigr)$ is a singleton and 
$\{x_n\}$ has a $\Delta$-convergent subsequence. 
See~\cite{MR2508878, MR2416076} 
for more details on $\Delta$-convergence. 
We know the following. 

\begin{lemma}[{\cite[Proposition~3.1]{MR3213144}}]
 \label{lem:Delta_conv}
 Let $X$ be a complete $\CAT(1)$ space 
 and $\{x_n\}$ a spherically bounded sequence in $X$. 
 If $\{d(z, x_n)\}$ is convergent 
 for all $z\in \omega_{\Delta} \bigl(\{x_n\}\bigr)$, 
 then $\{x_n\}$ is $\Delta$-convergent. 
\end{lemma}

Let $X$ be an admissible $\CAT(1)$ space 
and $f$ a function of $X$ into $\opclitvl{-\infty}{\infty}$. 
The function $f$ is said to be convex if 
\begin{align*}
 f\bigl(\alpha x \oplus (1-\alpha) y\bigr) 
 \leq \alpha f(x) + (1-\alpha) f(y)
\end{align*}
for all $x,y\in X$ and $\alpha \in \openitvl{0}{1}$. 
It is also said to be $\Delta$-lower semicontinuous if 
\begin{align*}
 f(p)\leq \liminf_{n\to \infty} f(x_{n})
\end{align*}
whenever $\{x_n\}$ is a sequence in $X$ 
which is $\Delta$-convergent to $p\in X$. 
We denote by $\Argmin_X f$ or 
$\Argmin_{y\in X}f(y)$ the set of all minimizers of $f$. 
A function $g$ of $X$ into $\clopitvl{-\infty}{\infty}$ 
is said to be concave if $-g$ is convex. 
We denote by $\Argmax_{X} g$ the set of all 
maximizers of $g$. 
See~\cite{MR1113394, MR3523548} 
on some examples of convex functions 
in $\CAT(1)$ spaces. 
We know the following. 
\begin{lemma}[{\cite[Lemma~3.1]{MR3463526}}]
 \label{lem:Delta-lsc-CAT1}
 Let $X$ be an admissible complete $\CAT(1)$ space and 
 $f$ a proper lower semicontinuous convex function of 
 $X$ into $\opclitvl{-\infty}{\infty}$. 
 Then $f$ is $\Delta$-lower semicontinuous. 
\end{lemma}

It is clear that if 
$A$ is a nonempty bounded subset of $\R$, 
$I$ is a closed subset of $\R$ which contains $A$, 
and $f$ is a continuous and nondecreasing real function on $I$, 
then $f(\sup A) =\sup f(A)$ and $f(\inf A) =\inf f(A)$. 
This implies the following. 

\begin{lemma}\label{lem:limsup-liminf}
 Let $I$ be a nonempty closed subset of $\R$, 
 $\{t_n\}$ a bounded sequence in $I$, 
 and $f$ a continuous real function on $I$. 
 Then the following hold.   
 \begin{enumerate}
  \item[(i)] If $f$ is nondecreasing, 
    then $f(\limsup_n t_n)= \limsup_n f(t_n)$; 
  \item[(ii)] if $f$ is nonincreasing, 
    then $f(\limsup_n t_n)= \liminf_n f(t_n)$.  
 \end{enumerate}
\end{lemma}

\section{Fundamental properties of resolvents in $\CAT(1)$ spaces}\label{sec:fund}

Let $X$ be an admissible complete $\CAT(1)$ space 
and $f$ a proper lower semicontinuous convex function 
of $X$ into $\opclitvl{-\infty}{\infty}$. 
It is known~\cite[Theorem~4.2]{MR3463526} that 
for each $x\in X$, there exists a unique 
$\hat{x}\in X$ such that 
\begin{align*}
 f(\hat{x}) +\tan d(\hat{x}, x)\sin d(\hat{x}, x) 
 =\inf _{y\in X} 
 \bigl\{f(y)+\tan d(y, x)\sin d(y, x)\bigr\}. 
\end{align*}
The resolvent $R_{f}$ of $f$ is defined by 
$R_f x =\hat{x}$ for all $x\in X$. 
In other words, $R_{f}$ is given by~\eqref{eq:resolvent-CAT1} 
for all $x\in X$. 
It is known~\cite[Theorems~4.2 and~4.6]{MR3463526} 
that $R_{f}$ is a well defined single valued mapping of $X$ into itself  
satisfying~\eqref{eq:FP-Min} and~\eqref{eq:SFSN-resolvent}. 

Using some techniques developed 
in the proof of~\cite[Theorem~4.6]{MR3463526}, 
we show the following fundamental result. 
The inequality~\eqref{eq:lem:res-fund-b} 
is a generalization of~\eqref{eq:SFSN-resolvent} 
and also a counterpart of~\cite[Lemma~3.1]{MR2780284} 
in the $\CAT(1)$ space setting. 

\begin{lemma}\label{lem:res-fund}
 Let $X$ be an admissible complete $\CAT(1)$ space, 
 $f$ a proper lower semicontinuous convex 
 function of $X$ into $\opclitvl{-\infty}{\infty}$, 
 $R_{\eta}$ the resolvent of $\eta f$ 
 for all $\eta >0$, 
 and $C_{\eta, z}$ the real number given by 
 $C_{\eta, z}=\cos d(R_{\eta}z,z)$ 
 for all $\eta >0$ and $z\in X$. 
 If $\lambda, \mu>0$ and $x,y\in X$, then the inequalities 
 \begin{align}
  \begin{split}\label{eq:lem:res-fund-a}
   &\sk{\frac{1}{C_{\lambda,x}^2}+1}
      d(R_{\lambda}x,R_{\mu}y)\bigl(
      C_{\lambda,x}\cos d(R_{\lambda}x,R_{\mu}y)
      -\cos d(R_{\mu}y,x)\bigr) \\
   &\geq 
     \lambda \bigl(f(R_{\lambda}x) - f(R_{\mu}y)\bigr) \sin
      d(R_{\lambda}x,R_{\mu}y)
  \end{split}
 \end{align}
 and
 \begin{align}
  \begin{split}\label{eq:lem:res-fund-b}
  &\Bigl(
    \lambda C_{\lambda, x}^2(1+C_{\mu,y}^2)C_{\mu,y}
     + 
    \mu C_{\mu, y}^2(1+C_{\lambda,x}^2)C_{\lambda,x}
    \Bigr)
    \cos d(R_{\lambda}x,R_{\mu}y) \\
  &\geq 
     \lambda C_{\lambda, x}^2(1+C_{\mu,y}^2)
     \cos d(R_{\lambda}x,y) +
    \mu C_{\mu, y}^2(1+C_{\lambda,x}^2)
    \cos d(R_{\mu}y,x)
  \end{split}
 \end{align}
 hold. 
\end{lemma}

\begin{proof}
 Let $\lambda, \mu >0$ and $x, y\in X$ be given. 
 Set $D=d(R_{\lambda}x, R_{\mu}y)$ and 
 \begin{align*}
  z_t=tR_{\mu}y\oplus (1-t)R_{\lambda}x
 \end{align*}
 for all $t\in \openitvl{0}{1}$. 
 By the definition of $R_{\lambda}$ and the convexity of $f$, 
 we have 
 \begin{align*}
  \begin{split}
   &\lambda f(R_{\lambda}x)+\tan d(R_{\lambda}x,x)\sin d(R_{\lambda}x,x) \\
   &\quad \leq \lambda f(z_t)+\tan d(z_t,x)\sin d(z_t,x) \\
   &\quad \leq t\lambda f(R_{\mu}y)
    +(1-t)\lambda f(R_{\lambda}x)+\tan d(z_t,x)\sin d(z_t,x)
  \end{split}
 \end{align*}
 and hence we have 
 \begin{align}
  \begin{split}\label{eq:lem:res-fund-c}
   & t\lambda\bigl( f(R_{\lambda}x)-f(R_{\mu}y) \bigr) \\
   & \leq \tan d(z_t, x) \sin d(z_t, x) - \tan d(R_{\lambda}x, x) \sin
   d(R_{\lambda}x, x) \\
  & = \sk{\frac{1}{\cos d(z_t, x)\cos d(R_{\lambda}x,x)} +1} 
 \bigl( \cos d(R_{\lambda}x, x) - \cos d(z_t, x) \bigr). 
  \end{split}
 \end{align} 
 On the other hand, Lemma~\ref{lem:KS-inequality} implies that 
 \begin{align}
  \begin{split}\label{eq:lem:res-fund-d}
  &\cos d(z_t, x) \sin d(R_{\mu}y, R_{\lambda}x) \\
  &\geq \cos d(R_{\mu}y,x)\sin \bigl(td(R_{\mu}y, R_{\lambda}x)\bigr) 
 +\cos d(R_{\lambda}x,x)\sin \bigl((1-t)d(R_{\mu}y, R_{\lambda}x)\bigr). 
  \end{split}
 \end{align}
 Using~\eqref{eq:lem:res-fund-c} 
 and~\eqref{eq:lem:res-fund-d}, we have  
 \begin{align*}
  \begin{split}
   &t\lambda\bigl( f(R_{\lambda}x)-f(R_{\mu}y) \bigr) \sin D \\
   &\leq \sk{\frac{1}{\cos d(z_t, x)\cos d(R_{\lambda}x,x)} +1} \\
   & \quad \times \Bigl[\cos d(R_{\lambda}x, x)
  \Bigl(\sin D - \sin \bigl((1-t)D \bigr)\Bigr) - \cos d(R_{\mu}y, x) \sin (tD) \Bigr] \\
   &= \sk{\frac{1}{\cos d(z_t, x)\cos d(R_{\lambda}x,x)} +1} \cdot 2\sin
   \sk{\frac{t}{2}D} \\
   & \quad \times \Biggl[ \cos d(R_{\lambda}x, x) \cos \left(\sk{1-\frac{t}{2}}D \right) 
  - \cos d(R_{\mu}y, x) \cos \sk{\frac{t}{2}D}\Biggr] 
  \end{split}
 \end{align*}
 and hence 
 \begin{align*}
  \begin{split}
   &\lambda\bigl( f(R_{\lambda}x)-f(R_{\mu}y) \bigr) \sin D \\
   &\leq \sk{\frac{1}{\cos d(z_t, x)\cos d(R_{\lambda}x,x)} +1} \cdot \frac{2}{t}\sin
   \sk{\frac{t}{2}D} \\
   & \quad \times \Biggl[ \cos d(R_{\lambda}x, x) \cos \left(\sk{1-\frac{t}{2}}D \right) 
  - \cos d(R_{\mu}y, x) \cos \sk{\frac{t}{2}D}\Biggr].  
  \end{split}
 \end{align*}
 Letting $t\downarrow 0$, we obtain 
 \begin{align*}
   \lambda\bigl( f(R_{\lambda}x)-f(R_{\mu}y) \bigr) \sin D 
   \leq \sk{\frac{1}{C_{\lambda, x}^2} +1} D
   \bigl(C_{\lambda, x} \cos D- \cos d(R_{\mu}y, x) \bigr).  
 \end{align*}
 Thus~\eqref{eq:lem:res-fund-a} holds. 

 If $D>0$, 
 then~\eqref{eq:lem:res-fund-a} implies that 
 \begin{align}
  \begin{split}\label{eq:lem:res-fund-e}
   &\mu C_{\mu, y}^2\bigl(1+C_{\lambda,x}^2\bigr)
      \bigl(
      C_{\lambda,x}\cos D
      -\cos d(R_{\mu}y,x)\bigr) \\
   &\geq 
     \frac{\lambda \mu C_{\lambda, x}^2C_{\mu, y}^2 }{D}
    \bigl(f(R_{\lambda}x) - f(R_{\mu}y)\bigr) \sin D
  \end{split}
 \end{align}
 and 
 \begin{align}
  \begin{split}\label{eq:lem:res-fund-f}
   &\lambda C_{\lambda, x}^2\bigl(1+C_{\mu,y}^2\bigr)
      \bigl(
      C_{\mu,y}\cos D
      -\cos d(R_{\lambda}x,y)\bigr) \\
   &\geq 
     \frac{\lambda \mu C_{\lambda, x}^2C_{\mu, y}^2 }{D}
    \bigl(f(R_{\mu}y) - f(R_{\lambda}x)\bigr) \sin D. 
  \end{split}
 \end{align}
 Adding~\eqref{eq:lem:res-fund-e} and~\eqref{eq:lem:res-fund-f}, 
 we obtain~\eqref{eq:lem:res-fund-b}.  
 It is obvious that the equality in~\eqref{eq:lem:res-fund-b} 
 holds in the case when $D=0$. 
\end{proof}

As a direct consequence of Lemma~\ref{lem:res-fund}, 
we obtain the following. 

\begin{corollary}\label{cor:res-fund}
 Let $X$, $f$, $\{R_{\eta}\}$, and  
 $\{C_{\eta, z}\}$ be the same as in Lemma~\ref{lem:res-fund}. 
 If $\lambda >0$, $x\in X$, and $y\in \Argmin_X f$, 
 then the inequalities 
 \begin{align}
 \begin{split}\label{eq:cor:res-fund-a}
  \frac{\pi}{2}\sk{\frac{1}{C_{\lambda, x}^2}+1} 
     \bigl(C_{\lambda, x} \cos d(y, R_{\lambda}x) - \cos d(y, x)\bigr) 
  \geq \lambda 
     \bigl(
      f(R_{\lambda} x) - f(y) 
     \bigr)   
 \end{split}
 \end{align}
 and
  \begin{align}\label{eq:cor:res-fund-b}
  C_{\lambda, x} \cos d(y,R_{\lambda}x) \geq \cos d(y,x)
 \end{align}
 hold. 
\end{corollary}

\begin{proof} 
 Let $\lambda >0$, $x\in X$, and $y\in \Argmin_X f$ be given. 
 Since $f(R_{\lambda}x)-f(y)\geq 0$ 
 and $\sin t \geq 2t/\pi$ for all $t\in [0,\pi/2]$, 
 it follows from~\eqref{eq:FP-Min} 
 and~\eqref{eq:lem:res-fund-a} that 
 \begin{align*}
  \begin{split}
  &\sk{\frac{1}{C_{\lambda, x}^2}+1} d(y, R_{\lambda}x) 
     \bigl(C_{\lambda, x} \cos d(y, R_{\lambda}x) - \cos d(y, x)\bigr) \\
  &\quad \geq \lambda 
     \bigl(
      f(R_{\lambda} x) - f(y) 
     \bigr) \cdot \frac{2d(y, R_{\lambda}x)}{\pi}. 
  \end{split}
 \end{align*}
 This implies that~\eqref{eq:cor:res-fund-a} holds when $d(y, R_{\lambda}x)>0$. 
 Note that the equality in~\eqref{eq:cor:res-fund-a} clearly holds 
 when $d(y, R_{\lambda}x)=0$. 
 It then follows from~\eqref{eq:cor:res-fund-a} that 
 \begin{align*}
  \begin{split}
\frac{\pi}{2} \sk{\frac{1}{C_{\lambda,x}^2} +1} 
\bigl(C_{\lambda, x} \cos d(y, R_{\lambda}x) - \cos d(y, x) \bigr) \geq 0 
  \end{split}
 \end{align*} 
 and hence~\eqref{eq:cor:res-fund-b} holds. 
\end{proof}

\section{The proximal point algorithm in $\CAT(1)$ spaces}\label{sec:PPA}

We need the following maximization theorem in the proof of 
Theorem~\ref{thm:ppa-CAT1-exist}.  

\begin{theorem}\label{thm:argmax}
 Let $X$ be an admissible complete $\CAT(1)$ space, 
 $\{z_n\}$ a spherically bounded sequence in $X$, 
 $\{\beta_n\}$ a sequence of positive real numbers 
 such that $\sum_{n=1}^{\infty}\beta_n =\infty$, 
 and $g$ the real function on $X$ defined by 
 \begin{align}
  g(y) = \liminf_{n\to \infty} \frac{1}{\sum_{l=1}^{n}\beta_l} 
  \sum_{k=1}^{n} \beta_k \cos d(y, z_k)
 \end{align} 
 for all $y\in X$. 
 Then $g$ is a concave and nonexpansive function 
 of $X$ into $[0,1]$ and $\Argmax_{X} g$ is a singleton. 
\end{theorem}

\begin{proof}
 Set $\sigma_n =\sum_{l=1}^{n}\beta_l$ for all $n\in \N$. 
 Since $X$ is admissible, we know that 
 \begin{align*}
  \frac{1}{\sigma_n}\sum_{k=1}^{n} \beta_k \cos d(y, z_k)\in \opclitvl{0}{1}
 \end{align*}
 and hence $g(y)\in [0,1]$ for all $y\in X$.  

 We next show that $g$ is concave and nonexpansive. 
 If $y_1,y_2\in X$ and $\alpha \in \openitvl{0}{1}$, 
 then it follows from Lemma~\ref{lem:CAT1-CMS} that 
 \begin{align*}
  \cos d\bigl(\alpha y_1 \oplus (1-\alpha)y_2, z_k\bigr) 
  \geq \alpha \cos d(y_1, z_k) + (1-\alpha) \cos d(y_2, z_k)
 \end{align*}
 for all $k\in \N$. This implies that 
 \begin{align*}
  &\frac{1}{\sigma_n} \sum_{k=1}^{n}\beta_k 
   \cos d\bigl(\alpha y_1 \oplus (1-\alpha)y_2, z_k\bigr) \\
  &\geq \frac{\alpha}{\sigma_n} \sum_{k=1}^{n}\beta_k \cos d(y_1, z_k) 
     + \frac{1-\alpha}{\sigma_n} \sum_{k=1}^{n}\beta_k  \cos d(y_2, z_k)
 \end{align*}
 for all $n\in \N$. 
 Taking the lower limit, we obtain 
 \begin{align*}
  g\bigl(\alpha y_1 \oplus (1-\alpha)y_2\bigr) 
  \geq \alpha g(y_1) + (1-\alpha) g(y_2)
 \end{align*}
 and hence $g$ is concave. 
 The nonexpansiveness of $t\mapsto \cos t$ 
 and the triangle inequality imply that 
 \begin{align*}
  \cos d(y_1, z_k) \leq d(y_1, y_2) + \cos d(y_2, z_k)
 \end{align*}
 for all $k\in \N$ and hence we have 
 \begin{align*}
  g(y_1) - g(y_2) \leq d(y_1,y_2)
 \end{align*}
 Similarly, we can see that $g(y_2) - g(y_1) \leq d(y_1,y_2)$. 
 Thus $g$ is nonexpansive. 

 We next show that $\Argmax_{X} f$ is nonempty. 
 The spherical boundedness of $\{z_n\}$ implies that 
 \begin{align*}
  0\leq \inf_{y\in X} \limsup_{n\to \infty} d(y, z_n) < \frac{\pi}{2}. 
 \end{align*}
 Since $t\mapsto \cos t$ is continuous and decreasing on $[0,\pi/2]$, 
 Lemma~\ref{lem:limsup-liminf} implies that 
 \begin{align}
  \begin{split}\label{eq:thm:argmax-a}
  0 &< \cos \sk{\inf_{y\in X} \limsup_{n\to \infty} d(y, z_n)} \\
     &= \sup_{y\in X} \cos \sk{\limsup_{n\to \infty} d(y, z_n)} 
  = \sup_{y\in X} \liminf_{n\to \infty} \cos d(y, z_n).    
  \end{split}
 \end{align}

 On the other hand, we can see that 
 \begin{align}
  \begin{split}\label{eq:thm:argmax-b}
  \liminf_{n\to \infty} \cos d(y, z_n) 
  \leq   \liminf_{n\to \infty} 
  \frac{1}{\sigma_n} \sum_{k=1}^{n}\beta_k \cos d(y, z_k) 
  \end{split}
 \end{align}
 for all $y\in X$. 
 In fact, setting $\gamma_n=\cos d(y, z_n)$ for all $n\in \N$, 
 we know that, for each $\gamma < \liminf_{n}\gamma_n$, 
 there exists $n_0\in \N$ such that 
 $\gamma < \gamma _k$ for all $k\geq n_0$. 
 Thus, if $p\in \N$, then we have 
 \begin{align*}
  \frac{1}{\sigma_{n_0+p}} \sum_{k=1}^{n_0+p} \beta_k \gamma _k 
  &=\frac{1}{\sigma_{n_0+p}} \sk{
   \sum_{k=1}^{n_0}\beta_k \gamma_k 
  + \sum_{k=n_0+1}^{n_0+p}\beta_k \gamma_k} \\
  &>\frac{1}{\sigma_{n_0+p}} \sk{
   \sum_{k=1}^{n_0}\beta_k \gamma_k 
  + \sum_{k=n_0+1}^{n_0+p}\beta_k \gamma} \\
  &=\frac{1}{\sigma_{n_0+p}} 
   \sum_{k=1}^{n_0}\beta_k \gamma_k 
  + \sk{1-\frac{\sigma_{n_0}}{\sigma_{n_0+p}}}\gamma.  
 \end{align*}
 Since $\sigma_n\to \infty$ as $n\to \infty$, we have 
 \begin{align*}
  \begin{split}
    \liminf_{n\to \infty} \frac{1}{\sigma_{n}} 
    \sum_{k=1}^{n} \beta_k \gamma _k 
   &= \liminf_{p\to \infty} \frac{1}{\sigma_{n_0+p}} 
        \sum_{k=1}^{n_0+p} \beta_k \gamma _k \\
   &\geq \liminf_{p\to \infty} \sk{\frac{1}{\sigma_{n_0+p}} 
         \sum_{k=1}^{n_0}\beta_k \gamma_k 
         + \sk{1-\frac{\sigma_{n_0}}{\sigma_{n_0+p}}}\gamma} = \gamma.    
  \end{split}
 \end{align*}
 Since $\gamma < \liminf_{n}\gamma_n$ is arbitrary, 
 we know that~\eqref{eq:thm:argmax-b} holds. 

 By~\eqref{eq:thm:argmax-a} and~\eqref{eq:thm:argmax-b}, 
 we have 
 \begin{align}\label{eq:thm:argmax-c}
  0 < \sup_{y\in X}\liminf_{n\to \infty} \cos d(y, z_n) 
  \leq \sup_{y\in X} g(y)=:l. 
 \end{align}
 By the definition of $l$, there exists a sequence 
 $\{y_n\}$ in $X$ such that 
 $g(y_{n})\leq g(y_{n+1})$ for all $n\in \N$ 
 and $g(y_n)\to l$ as $n\to \infty$. 
 If $m\geq n$, 
 then Lemma~\ref{lem:KS-inequality-Mid} implies that 
 \begin{align*}
   \cos d\Bigl(\frac{1}{2}y_n \oplus \frac{1}{2}y_m, z_k\Bigr) 
      \cos \Bigl(\frac{1}{2}d(y_n,y_m)\Bigr) 
   \geq \frac{1}{2} \cos d(y_n, z_k) 
       + \frac{1}{2} \cos d(y_m, z_k) 
 \end{align*}
 for all $k\in \N$. This gives us that 
 \begin{align*}
  \begin{split}
   g\Bigl(\frac{1}{2}y_n \oplus \frac{1}{2}y_m\Bigr) 
     \cos \Bigl(\frac{1}{2}d(y_n,y_m)\Bigr) 
   \geq \frac{1}{2} g(y_n) + \frac{1}{2} g(y_m). 
  \end{split}
 \end{align*}
 Since $l=\sup g(X)$ and $g(y_n)\leq g(y_m)$, we then obtain 
 \begin{align}
  \begin{split}\label{eq:thm:argmax-d}
   l\cos \Bigl(\frac{1}{2}d(y_n,y_m)\Bigr) 
   \geq \frac{1}{2} g(y_n) + \frac{1}{2} g(y_m) \geq g(y_n).  
  \end{split}
 \end{align} 
 Noting that~\eqref{eq:thm:argmax-c} implies that $0<l\leq 1$, 
 we have 
 \begin{align}
  \begin{split}\label{eq:thm:argmax-e}
    d(y_n,y_m) \leq 2\arccos \frac{g(y_n)}{l}
  \end{split}
 \end{align} 
 whenever $m\geq n$. 
 Since $g(y_n)/l\to 1$ as $n\to \infty$, 
 the right hand side of~\eqref{eq:thm:argmax-e} 
 converges to $0$. 
 Thus $\{y_n\}$ is a Cauchy sequence in $X$.  
 Since $X$ is complete, the sequence $\{y_n\}$ converges to some $p\in X$. 
 By the continuity of $g$ and the choice of $\{y_n\}$, 
 we obtain 
 \begin{align*}
  g(p)=\lim_{n\to \infty} g(y_n)=l. 
 \end{align*}
 Thus $p$ is an element of $\Argmax_{X} g$. 
 
 We finally show that $\Argmax_{X} g$ consists of one point. 
 Suppose that $p$ and $q$ are elements of $\Argmax_{X} g$. 
 As in the proof of~\eqref{eq:thm:argmax-d}, we can see that 
 \begin{align*}
   l\cos \Bigl(\frac{1}{2}d(p, q)\Bigr) 
  \geq \frac{1}{2} g(p) + \frac{1}{2} g(q) =l.  
 \end{align*}
 Since $l>0$, we then obtain 
 $\cos \bigl(d(p, q)/2\bigr) =1$. 
 Consequently, we have $p=q$. 
\end{proof}

Now, we are ready to give the proofs of 
Theorems~\ref{thm:ppa-CAT1-exist} and~\ref{thm:ppa-CAT1-conv}.  
In these proofs, we denote by $R_{\eta}$ and $C_{\eta, z}$ 
the resolvent of $\eta f$ for all $\eta >0$ 
and the real number given by 
$C_{\eta, z}=\cos d(R_{\eta}z,z)$ for all $\eta>0$ and $z\in X$, 
respectively. 

\begin{proof}[The proof of Theorem~\ref{thm:ppa-CAT1-exist}]
 We first show the if part. 
 Suppose that $\{x_n\}$ is spherically bounded and 
 \begin{align}\label{eq:thm:ppa-CAT1-exist-a1}
  \sup_{n} d(x_{n+1},x_{n})<\frac{\pi}{2}. 
 \end{align} 
 Set 
 \begin{align*}
  \beta_n= \frac{\lambda_n C_{\lambda_n, x_n}^2}{1+C_{\lambda_n,x_n}^2} 
  \quad \textrm{and} \quad 
  \sigma_n =\sum_{k=1}^{n} \beta _k
 \end{align*} 
 for all $n\in \N$. 
 It is obvious that $\beta_n >0$ for all $n\in \N$. 
 It also follows from~\eqref{eq:thm:ppa-CAT1-exist-a1} that 
 \begin{align*}
  0 < \cos \sk{\sup_{n} d(x_{n+1},x_n)} 
     =\inf_{n} \cos d(x_{n+1},x_n) =\inf_{n} C_{\lambda_n,x_n} 
 \end{align*}
 and hence it follows from 
 \begin{align*}
  \beta_n \geq \frac{\lambda_n C_{\lambda_n, x_n}^2}{2} 
  \quad \textrm{and} \quad 
  \sum_{n=1}^{\infty}\lambda_n=\infty
 \end{align*} 
 that $\sum_{n=1}^{\infty}\beta_n=\infty$. 
 Thus Theorem~\ref{thm:argmax} ensures that 
 the real function $g$ on $X$, which is defined by 
 \begin{align*}
  g(y) = \liminf_{n\to \infty} 
  \frac{1}{\sigma_n} \sum_{k=1}^{n} \beta_k \cos d(y, x_{k+1})
 \end{align*}
 for all $y\in X$, 
 has a unique maximizer $p$ on $X$. 

 Let $\mu$ be a positive real number. 
 By~\eqref{eq:lem:res-fund-b}, we have 
 \begin{align*}
  \begin{split}
  &\Bigl(
    \lambda_k C_{\lambda_k, x_k}^2(1+C_{\mu,p}^2)
     + 
    \mu C_{\mu, p}^2(1+C_{\lambda_k,x_k}^2)
    \Bigr)
    \cos d(x_{k+1},R_{\mu}p) \\
  &\geq 
     \lambda_k C_{\lambda_k, x_k}^2(1+C_{\mu,p}^2)
     \cos d(x_{k+1},p) +
    \mu C_{\mu, p}^2(1+C_{\lambda_k,x_k}^2)
    \cos d(R_{\mu}p,x_k)
  \end{split}
 \end{align*}
 and hence 
 \begin{align}
  \begin{split}\label{eq:thm:ppa-CAT1-exist-a3}
  &\frac{\lambda_k C_{\lambda_k, x_k}^2}{1+C_{\lambda_k,x_k}^2}
    \cos d(x_{k+1},R_{\mu}p) \\
  &\geq 
     \frac{\lambda _k C_{\lambda_k, x_k}^2}{1+C_{\lambda_k,x_k}^2}
     \cos d(x_{k+1},p) 
  +
    \frac{\mu C_{\mu, p}^2}{1+C_{\mu,p}^2}
    \bigl(\cos d(R_{\mu}p,x_k) - \cos d(R_{\mu}p,x_{k+1})\bigr)
  \end{split}
 \end{align}
 for all $k\in \N$. 
 Summing up~\eqref{eq:thm:ppa-CAT1-exist-a3} 
 with respect to $k\in \{1,2,\dots , n\}$, we have 
 \begin{align*}
  \begin{split}
  &\frac{1}{\sigma_n}
  \sum_{k=1}^{n} \beta_k \cos d(x_{k+1},R_{\mu}p) \\
  &\geq 
     \frac{1}{\sigma_n}
  \sum_{k=1}^{n} \beta_k \cos d(x_{k+1},p) 
  +
    \frac{\mu C_{\mu, p}^2}{1+C_{\mu,p}^2}
    \cdot \frac{\cos d(R_{\mu}p,x_1) - \cos d(R_{\mu}p,x_{n+1})}{\sigma_n}
  \end{split}
 \end{align*} 
 for all $n\in \N$. 
 Since $\sigma_{n}\to \infty$ as $n\to \infty$, we obtain 
 \begin{align*}
  g(R_{\mu} p) 
 &=\liminf_{n\to \infty} \frac{1}{\sigma_n}
  \sum_{k=1}^{n} \beta_k \cos d(x_{k+1},R_{\mu}p) \\
 &\geq \liminf_{n\to \infty} 
\frac{1}{\sigma_n}
  \sum_{k=1}^{n} \beta_k \cos d(x_{k+1},p) =g(p). 
 \end{align*}
 Then it follows from $\Argmax_{X}g=\{p\}$ that $R_{\mu}p=p$. 
 By~\eqref{eq:FP-Min}, we know that 
 \begin{align*}
  \Fix(R_{\mu}) = \Argmin_X \mu f = \Argmin_X f
 \end{align*}
 and hence we conclude that $p$ is an element of $\Argmin_X f$.  

 We next show the only if part. 
 Suppose that $\Argmin_X f$ is nonempty and 
 let $u$ be an element of $\Argmin_X f$. 
 It follows from~\eqref{eq:cor:res-fund-b} that 
 \begin{align}\label{eq:thm:ppa-CAT1-exist-b1}
  \begin{split}
  \min\bigl\{\cos d(x_{n+1}, x_{n}), \cos d(u,x_{n+1})\bigr\}
  &\geq \cos d(x_{n+1}, x_{n}) \cos d(u,x_{n+1}) \\
  &\geq \cos d(u,x_{n}). 
  \end{split}
 \end{align}
 The admissibility of $X$ and~\eqref{eq:thm:ppa-CAT1-exist-b1} 
 imply that 
 \begin{align}\label{eq:thm:ppa-CAT1-exist-b2}
  \max\bigl\{d(x_{n+1}, x_{n}), d(u, x_{n+1})\bigr\} 
 \leq d(u, x_{n})\leq d(u, x_{1})<\frac{\pi}{2}
 \end{align}
 and hence $\{x_n\}$ is spherically bounded 
 and $\sup_{n}d(x_{n+1},x_{n}) <\pi/2$. 
\end{proof}

\begin{proof}[The proof of Theorem~\ref{thm:ppa-CAT1-conv}]
 We first show~(i).  
 Set $l=\sup_{n}d(x_{n+1},x_{n})$.  
 By Theorem~\ref{thm:ppa-CAT1-exist}, we know that 
 $\{x_n\}$ is spherically bounded and 
 $l <\pi/2$. 
 Letting 
 \begin{align*}
  K=\frac{1}{\cos^2 l} +1, 
 \end{align*}
 we have 
 \begin{align}\label{eq:thm:ppa-CAT1-conv-temp-a}
  \frac{1}{\cos ^2d(x_{n+1},x_{n})} + 1 \leq K
 \end{align}
 for all $n\in \N$. 
 Let $u$ be an element of $\Argmin_X f$. 
 By the definitions of $R_{\lambda_n}$ 
 and $\{x_n\}$, we know that 
 \begin{align}
  \begin{split}\label{eq:thm:ppa-CAT1-conv-c5}
  f(u) 
   &\leq f(x_{n+1}) \\
   &\leq f(x_{n+1}) 
   + \frac{1}{\lambda _n} \tan d(x_{n+1}, x_{n}) 
      \sin d(x_{n+1}, x_{n}) \\
   &\leq f(x_{n})  
  \end{split} 
 \end{align}
 for all $n\in \N$. 
 On the other hand, it follows from~\eqref{eq:cor:res-fund-a} that 
 \begin{align}
  \begin{split}\label{eq:thm:ppa-CAT1-conv-c1}
  &\lambda _{n}\bigl(
        f(x_{n+1}) - f(u) \bigr) \\
  &\leq \frac{\pi}{2}\sk{\frac{1}{\cos ^2d(x_{n+1},x_{n})}+1} 
     \bigl(\cos d(u, x_{n+1}) -\cos d(u, x_{n})\bigr) 
  \end{split}
 \end{align}
 for all $n\in \N$. 
 If $n\in \N$ and $k\in \{1,2,\dots, n\}$, then 
 it follows from~\eqref{eq:thm:ppa-CAT1-conv-temp-a},~\eqref{eq:thm:ppa-CAT1-conv-c5}, 
 and~\eqref{eq:thm:ppa-CAT1-conv-c1} that 
 \begin{align*}
  \begin{split}
  \lambda _{k}\bigl( f(x_{n+1}) - f(u) \bigr) 
  &\leq \lambda _{k}\bigl( f(x_{k+1}) - f(u) \bigr) \\
  &\leq \frac{K\pi}{2}\bigl(\cos d(u, x_{k+1}) -\cos d(u, x_{k})\bigr). 
  \end{split}
 \end{align*}
Hence we obtain 
 \begin{align*}
 \begin{split}
   \bigl( f(x_{n+1}) - \inf f(X) \bigr) \sum_{k=1}^{n}\lambda _k 
   &\leq \frac{K\pi}{2} 
     \bigl(\cos d(u, x_{n+1}) -\cos d(u, x_{1})\bigr) \\
   &\leq \frac{K\pi}{2} \bigl(1-\cos d(u, x_{1})\bigr). 
 \end{split}
 \end{align*}
 Letting $C=K\pi/2$, we obtain the desired inequality. 

 We finally show~(ii). 
 Since $\sum_{n=1}^{\infty}\lambda_n=\infty$, 
 it follows from~(i) that 
 \begin{align}\label{eq:thm:ppa-CAT1-conv-c7}
  \lim_{n\to \infty} f(x_n) =\inf f(X). 
 \end{align}
 We then show that $\{d(z, x_{n})\}$ is convergent 
 for all $z \in \omega_{\Delta}\bigl(\{x_n\}\bigr)$. 
 If $z$ is an element of $\omega_{\Delta}\bigl(\{x_n\}\bigr)$, 
 then we have a subsequence $\{x_{n_i}\}$ of $\{x_n\}$ 
 which is $\Delta$-convergent to $z$. 
 By Lemma~\ref{lem:Delta-lsc-CAT1} and~\eqref{eq:thm:ppa-CAT1-conv-c7}, 
 we obtain 
 \begin{align*}
  f(z)\leq \liminf_{i\to \infty} f(x_{n_i}) 
  =\lim_{n\to \infty} f(x_n) = \inf f(X) 
 \end{align*}
 and hence $z$ is an element of $\Argmin_X f$. 
 Thus $\omega_{\Delta}\bigl(\{x_n\}\bigr)$ 
 is a subset of $\Argmin_X f$. 
 It also follows from~\eqref{eq:thm:ppa-CAT1-exist-b2} that 
 $\{d(z, x_{n})\}$ is convergent. 
 Thus, Lemma~\ref{lem:Delta_conv} implies that 
 $\{x_{n}\}$ is $\Delta$-convergent to some $x_{\infty}\in X$. 
 This gives us that 
 \begin{align*}
  \{x_{\infty}\}=\omega_{\Delta}\bigl(\{x_n\}\bigr) 
 \subset \Argmin_X f. 
 \end{align*} 
 Consequently, $\{x_n\}$ is 
 $\Delta$-convergent to an element of $\Argmin_X f$. 
\end{proof}

\section{Three corollaries of Theorems~\ref{thm:ppa-CAT1-exist} and~\ref{thm:ppa-CAT1-conv}}\label{sec:cor}

Using Theorems~\ref{thm:ppa-CAT1-exist} 
and~\ref{thm:ppa-CAT1-conv}, 
we obtain the following three corollaries. 

\begin{corollary}\label{cor:conv-CAT1-single}
 Let $X$ be an admissible complete $\CAT(1)$ space, 
 $f$ a proper lower semicontinuous convex 
 function of $X$ into $\opclitvl{-\infty}{\infty}$, 
 $R_{f}$ the resolvent of $f$, 
 and $x$ an element of $X$. 
 Then the following hold. 
 \begin{enumerate}
  \item[(i)] The set $\Argmin_X f$ is nonempty if and only if 
 $\{R_{f}^nx\}$ is spherically bounded and 
  $\sup_{n}d(R_{f}^{n+1}x,R_{f}^{n}x)< \pi /2$; 
  \item[(ii)] if $\Argmin_X f$ is nonempty, then 
 $\{R_{f}^nx\}$ is $\Delta$-convergent 
 to an element of $\Argmin_X f$ 
 and there exists a positive real number $C$ such that 
 \begin{align*}
  f(R_f^{n}x) - \inf f(X) \leq 
 \frac{C}{n} \bigl(1-\cos d(u, x_1)\bigr) 
 \end{align*}
 for all $u\in \Argmin_X f$ and $n\in \N$. 
 \end{enumerate}
 \end{corollary}

\begin{proof}
 Set $\lambda_n =1$ for all $n\in \N$ 
 and let $\{x_n\}$ be a sequence defined by 
 $x_1=x$ and $x_{n+1}=R_{\lambda_n f} x_n$ 
 for all $n\in \N$.  
 Then we have $x_n=R_f^{n-1}x$ and $\sum_{k=1}^{n}\lambda_k=n$ 
 for all $n\in \N$.  
 Thus Theorems~\ref{thm:ppa-CAT1-exist} 
 and~\ref{thm:ppa-CAT1-conv} imply the conclusion. 
\end{proof}

\begin{remark}
 The part~(i) of Corollary~\ref{cor:conv-CAT1-single} 
 is related to~\cite[(i) of Theorem~7.1]{MR3463526}, 
 where it is shown that $\Argmin_X f$ is nonempty 
 if and only if there exists $w \in X$ such that 
 \begin{align*}
  \limsup_{n\to \infty} d(R_{f}y, R_{f}^nw) < \frac{\pi}{2}
 \end{align*} 
 for all $y\in X$. 
 On the other hand, 
 the former part of~(ii) is 
 a refinement of~\cite[(ii) of Theorem~7.1]{MR3463526}, 
 where it is additionally assumed that 
 \begin{align*}
  \limsup_{n\to \infty} d(R_{f}y, y_n) < \frac{\pi}{2}
 \end{align*} 
 whenever $\{y_n\}$ is a sequence in $X$ 
 which is $\Delta$-convergent to $y\in X$. 
\end{remark}

\begin{corollary}\label{cor:conv-CAT1-HS}
 Let $X$ be a nonempty closed convex admissible subset 
 of a Hilbert sphere $(S_H, \rho_{S_H})$, 
 both $f$ and $\{\lambda _n\}$ the same as in 
 Theorem~\ref{thm:ppa-CAT1-exist}, 
 and $\{x_n\}$ a sequence defined by 
 $x_1\in X$ and 
 \begin{align*}
  x_{n+1} = \Argmin_{y\in X} 
 \ck{f(y)+\frac{1}{\lambda_n} \tan \rho_{S_H}(y, x_{n})\sin \rho_{S_H}(y, x_{n})}
 \end{align*}
 for all $n\in \N$. 
 Then the following hold. 
 \begin{enumerate}
  \item[(i)] The set $\Argmin_X f$ is nonempty if and only if 
 $\{x_n\}$ is spherically bounded and 
  $\sup_{n}\rho_{S_H}(x_{n+1},x_{n})< \pi/2$; 
  \item[(ii)] if $\Argmin_X f$ is nonempty, then 
 $\{x_n\}$ is $\Delta$-convergent 
 to an element of $\Argmin_X f$ 
 and there exists a positive real number $C$ such that 
 \begin{align*}
  f(x_{n+1}) - \inf f(X) \leq 
  \frac{C}{\sum_{k=1}^{n}\lambda _k} \bigl(1-\cos \rho_{S_H}(u, x_1)\bigr)
 \end{align*}
 for all $u\in \Argmin_X f$ and $n\in \N$. 
 \end{enumerate}
 \end{corollary}

\begin{proof}
 Since $(X, \rho_{S_H}\vert_{X\times X})$ is 
 an admissible complete $\CAT(1)$ space, 
 Theorems~\ref{thm:ppa-CAT1-exist} 
 and~\ref{thm:ppa-CAT1-conv} imply the conclusion. 
\end{proof}

\begin{corollary}\label{cor:ppa-CATk}
 Let $\kappa$ be a positive real number, 
 $X$ a complete $\CAT(\kappa)$ space such that 
 $d(v,v')<\pi/(2\sqrt{\kappa})$ for all $v,v'\in X$, 
 $f$ a proper lower semicontinuous convex 
 function of $X$ into $\opclitvl{-\infty}{\infty}$, 
 $\{\lambda _n\}$ a sequence of positive real numbers 
 such that $\sum _{n=1}^{\infty} \lambda _n =\infty$, 
 and $\{x_n\}$ a sequence defined by 
 $x_1\in X$ and 
 \begin{align*}
  x_{n+1} = \Argmin_{y\in X} 
 \ck{f(y)+\frac{1}{\lambda_n} \tan \bigl(\sqrt{\kappa}d(y, x_{n})\bigr)
 \sin \bigl(\sqrt{\kappa}d(y, x_{n})\bigr)}
 \end{align*}
 for all $n\in \N$. 
 Then the following hold. 
 \begin{enumerate}
  \item[(i)] The set $\Argmin_X f$ is nonempty if and only if 
 \begin{align*}
  \inf_{y\in X}\limsup_{n\to \infty} d(y, x_n)<\frac{\pi}{2\sqrt{\kappa}}
  \quad \textrm{and} \quad  
  \sup_{n}d(x_{n+1},x_{n})<\frac{\pi}{2\sqrt{\kappa}}; 
 \end{align*}
  \item[(ii)] if $\Argmin_X f$ is nonempty, then 
 $\{x_n\}$ is $\Delta$-convergent 
 to an element of $\Argmin_X f$ 
 and there exists a positive real number $C$ such that 
 \begin{align*}
  f(x_{n+1}) - \inf f(X) \leq \frac{C}{\sum_{k=1}^{n}\lambda _k}
  \Bigl(1-\cos \bigl(\sqrt{\kappa}d(u, x_1)\bigr)\Bigr)
 \end{align*}
 for all $u\in \Argmin_X f$ and $n\in \N$. 
 \end{enumerate}
\end{corollary}

\begin{proof}
 Since $(X, d)$ is a complete $\CAT(\kappa)$ space 
 if and only if $(X, \sqrt{\kappa}d)$ is a complete $\CAT(1)$ space, 
 Theorems~\ref{thm:ppa-CAT1-exist} 
and~\ref{thm:ppa-CAT1-conv} imply the conclusion. 
\end{proof}

\section*{Note added in proof}

We finally note that 
Theorem~\ref{thm:ppa-CAT1-exist}  
and the part~(ii) of Theorem~\ref{thm:ppa-CAT1-conv} 
were announced 
in the talk~\cite{Kohsaka-ACFPTO16-talk} 
based on~\cite{MR3463526, KimuraKohsaka-LNA16} 
and the present paper. 
On the other hand, Esp{\'{\i}}nola and Nicolae~\cite{EspinolaNicolae-arXiv16} 
studied the proximal point algorithm 
and the splitting proximal point algorithm 
for convex functions in $\CAT(\kappa)$ spaces with positive $\kappa$.  
The $\Delta$-convergence result in 
the part~(ii) of Corollary~\ref{cor:ppa-CATk} 
was also found in~\cite{EspinolaNicolae-arXiv16}. 
In~\cite{Espinola-Personal}, 
the authors in this paper and 
the authors in~\cite{EspinolaNicolae-arXiv16} 
confirmed that there is the overlapping stated above 
and that these two papers were independently written. 

\section*{Acknowledgment}
The authors would like to express their sincere appreciation 
to Professor Rafael Esp{\'{\i}}nola 
for kindly letting them know 
his joint work~\cite{EspinolaNicolae-arXiv16} 
with Professor Adriana Nicolae. 
This work was supported by JSPS KAKENHI Grant Numbers 
15K05007 and 25800094.

\end{document}